\documentclass[11pt,reqno]{amsart}
\usepackage{amsmath,amsthm,amssymb,amscd,url,enumerate,mathtools}
\usepackage{graphicx}
\usepackage[margin=1.5in]{geometry}
\usepackage{afterpage}
\usepackage{tikz}
\usepackage{todonotes}
\usepackage[all]{xy}
\usepackage[colorlinks=true]{hyperref}
\usepackage{tabularx}

\newcommand{\dnd}{\nmid}

\newcommand{\TITLE}{A family of monogenic $S_4$ quartic fields arising from elliptic curves}
\newcommand{\TITLERUNNING}{Monogenic fields from elliptic curves}

\theoremstyle{plain}
\newtheorem{theorem}{Theorem}

\newtheorem{proposition}[theorem]{Proposition}
\newtheorem{lemma}[theorem]{Lemma}

\theoremstyle{definition}
\theoremstyle{remark}

\numberwithin{theorem}{section}

  {\end{list}}

  {\end{list}}
\newcommand{\tightoverset}[2]{%
  \mathop{#2}\limits^{\vbox to -.5ex{\kern-1.05ex\hbox{$#1$}\vss}}}
\newcommand{\NN}{\mathbb{N}}

\newcommand{\QQ}{\mathbb{Q}}

\newcommand{\ZZ}{\mathbb{Z}}

\newcommand{\Disc}{\operatorname{Disc}}

\renewcommand{\gcd}{{\operatorname{gcd}}}

\newcommand{\MOD}[1]{~(\textup{mod}~#1)}
\renewcommand{\pmod}{\MOD}

\newcommand{\res}{\operatornamewithlimits{res}}

\renewcommand{\setminus}{\smallsetminus}

\newcommand{\red}{\operatorname{red}}
\newcommand{\bbf}{\mathbb{F}}
\newcommand{\ind}{\operatorname{ind}}

\title[\TITLERUNNING]{\vspace*{-1.3cm} \TITLE}
\date{\today}

\author[T. Alden Gassert]{T. Alden Gassert}
\address{%
Department of Mathematics and Computer Science, Hobart and William Smith Colleges, 300 Pulteney St, Geneva, New York 14456}
\email{gassert@hws.edu}

\author[Hanson Smith]{Hanson Smith}
\address{%
Department of Mathematics, University of Colorado,
Campus Box 395, Boulder, Colorado 80309-0395}
\email{hanson.smith@colorado.edu}

\author[Katherine E. Stange]{Katherine E. Stange}
\address{%
Department of Mathematics, University of Colorado,
Campus Box 395, Boulder, Colorado 80309-0395}
\email{kstange@math.colorado.edu}


\thanks{%
The second and third author have been supported by NSF EAGER DMS-1643552.  The third author has also been supported by NSF CAREER CNS-1652238, and NSA Young Investigator Grants H98230-16-1-0040 and H98230-14-1-0106.  The United States Government is authorized to reproduce and distribute reprints notwithstanding any copyright notation herein.
}

\begin{document}

\maketitle

%%%%%%%%%%%%%%%%%%%%%%%%%%%%%%%%%%%%%%%%%%%%%%%%%%%%%%%%%%%%%%%%%%%%%%
%%% Text (non-TeX) Abstract
%%%%%%%%%%%%%%%%%%%%%%%%%%%%%%%%%%%%%%%%%%%%%%%%%%%%%%%%%%%%%%%%%%%%%%

\begin{abstract}
  We consider partial torsion fields (fields generated by a root of a division polynomial) for elliptic curves.  By analysing the reduction properties of elliptic curves, and applying the Montes Algorithm, we obtain information about the ring of integers.  In particular, for the partial $3$-torsion fields for a certain one-parameter family of non-CM elliptic curves, we describe a power basis.  As a result, we show that the one-parameter family of quartic $S_4$ fields given by $T^4 - 6T^2 - \alpha T - 3$ for $\alpha \in \ZZ$ such that $\alpha \pm 8$ are squarefree, are monogenic. 
\end{abstract}

\section{Introduction}

Consider the following result.

\begin{theorem}
  \label{thm:field-family}
  Suppose that $\alpha \pm 8$ is squarefree, where $\alpha \in \ZZ$.  Then the field $K_\alpha = \QQ(\theta)$ where $\theta$ is a root of the irreducible polynomial $T^4 - 6T^2 - \alpha T - 3$ has ring of integers $\ZZ[\theta]$; in other words, $K_\alpha$ is a quartic monogenic field.
\end{theorem}

The discriminant of this polynomial, and hence the field $\QQ(\theta)$, is $-27(\alpha-8)^2(\alpha+8)^2$.  We do not doubt that monogenicity can be deduced by classical computations, but the novelty of this paper is our method:  we discover this family of quartic fields as partial torsion fields (fields generated by a root of a division polynomial) of a particular family of elliptic curves, and deduce monogenicity by reference to reduction properties of the elliptic curve.  In particular, we prove the following.

\begin{theorem}
  \label{thm:ec-family}
  Let $E$ be an elliptic curve defined over $\QQ$, such that some twist $E'$ of $E$ has a $4$-torsion point defined over $\QQ$.  Then the following are equivalent:
  \begin{enumerate}
    \item \label{item:red} $E'$ has reduction types $I^*_1$ and $I_1$ only;
    \item \label{item:sqf} $E$ has $j$-invariant with squarefree denominator except a possible factor of 4.
    \item \label{item:j} $E$ has $j$-invariant $j = \frac{(\alpha^2 - 48)^3}{(\alpha - 8)(\alpha + 8)}$, where $\alpha \in \ZZ$, $\alpha \pm 8$ are squarefree.
  \end{enumerate}
  Let $K_n$ be the field defined by adjoining the $x$-coordinate of an $n$-torsion point of $E$.  
  If any of the above hypotheses holds, then
    $K_3$ is monogenic with a generator given by a root of $T^4 - 6T^2 - \alpha T - 3$.  In particular, the field $K_3$ has discriminant $-27(\alpha-8)^2(\alpha+8)^2$.
\end{theorem}
Some examples of small values of $\alpha$ for which $K_3$ is monogenic are:
$$\pm2, \pm3, \pm5, \pm6, \pm7, \pm9, \pm11, \pm13, \pm14, \pm15, \pm18, \pm21, \pm22, \pm23, \pm25.$$
The methods used in the proof turn information about reduction properties of an elliptic curve into information about the index $[\mathcal{O}_{\QQ(\theta)} : \ZZ[\theta]]$ where $\theta$ is a special value of an elliptic function (namely, a zero of a division polynomial).  Theorem \ref{thm:ec-family} is meant primarily to showcase our methods.  A more detailed analysis using these same methods can provide bounds and even exact formulae for the discriminants of partial torsion fields in general.  This will be described in a follow-up paper by the second author.  

In fact, Fleckinger and V\'erant studied the number fields of Theorem \ref{thm:field-family}, motivated by their status as partial torsion fields \cite{FV}.  However, as they write, ``We note that the arithmetic of elliptic curves is not used once we have these polynomials.''  They describe a basis for the ring of integers in general (which is not a power basis), and show that they are quartic $S_4$ fields.  See Section \ref{sec:ant}.

There is an abundance of literature on both monogenic number fields and number fields obtained by adjoining torsion points of elliptic curves.  Monogenicity is rare:  while our favourites, including quadratic fields, and the cyclotomic fields, are monogenic, it is known for example that almost all abelian extensions of $\QQ$ with degree coprime to $6$ are non-monogenic \cite{Gras}.  For an in-depth bibliography of monogenicity, see Narkiewicz \cite[pp. 79-81]{Nark} and the book of Ga\'al \cite{MR1896601}, and for fundamental algorithmic work, see Gy\H ory \cite{MR0437489}. We content ourselves here with listing a few recent works concerning monogenic quartic fields. In \cite{spear}, Spearman describes an infinite family of $A_4$ monogenic fields arising from $x^4+18x^2-4tx +t^2+81$ when $t(t^2+81)$ is squarefree.  The $D_8$ fields are studied by Kable \cite{MR1688180} and Huard, Spearman, and Williams \cite{MR1321725}. While the pure quartic case is investigated by Funakura, who finds infinitely many monogenic fields \cite{MR0779772}.  Fleckinger and V\'erant also have a monogenic family which appears to be $D_8$ \cite[(2)]{FV}. In \cite{tangras}, Gras and Tano\'e list necessary and sufficient conditions for certain biquadratic extensions of $\QQ$ to be monogenic; Motoda constructs an infinite family \cite{MR2017249}.  It is also known that infinitely many quartic cyclic fields are non-monogenic, by work of Motoda, Nakahara, Shah and Uehara \cite{MR2605782} and also Olajos \cite{MR2169516}.  As for $S_4$ fields, little is known; however, B\'erczes, Evertse and Gy\H ory restrict the multiply monogenic orders in such fields \cite{MR3114010}.  See the experimental data in Section \ref{sec:exp} for three more families of quartic fields which appear to be monogenic. 

The field over which the $n$-torsion points of an elliptic curve are defined is often denoted $\QQ(E[n])$ and plays a crucial role in the study of elliptic curves and their Galois representations.  It is often referred to as a division field or a torsion field.  For a survey, see \cite{MR1836119}.  In general, the discriminants of such fields are not known, although there has been some work on their ramification \cite{MR2028514, MR3438391, MR3501021}. In the case when $n$ is prime the different has been computed \cite{KrausCali, Kraus}. In the case of $3$-division fields, generators, Galois groups and subfields have been very explicitly described \cite{BP2012}; see \cite{BP2016} for higher order. However, little similar work has been done on the subfields defined by division polynomials.

The Fueter polynomial we study arises from changing coordinates to the Fueter form of an elliptic curve:  this choice has a history in explicit class field theory.  Specifically, in \cite{ct87}, Cassou-Nogu\`{e}s and Taylor pursue Kronecker's Jugendtraum for certain ray class fields of imaginary quadratic fields. They study elliptic curves with complex multiplication and good reduction away from $2$. Let $K$ be an imaginary quadratic field with discriminant $d_K<-4$ and suppose $2$ splits in $K$. If $I$ is any $\mathcal{O}_K$ ideal, let $K(I)$ denote the ray class field of $K$ mod $I$. Now suppose $\xi$ is an odd $\mathcal{O}_K$ ideal. Cassou-Nogu\`{e}s and Taylor show that $\mathcal{O}_{K(4\xi)}$ is monogenic over $\mathcal{O}_{K(4)}$, using special values of the coordinates of the Fueter form.  

Although the methods and the class of monogenic fields found in \cite{ct87} differ, we adopt their use of the Fueter form to access special values of an elliptic function.  
It is remarkable that in the non-CM case, these special values still seem to offer some advantage in describing partial torsion fields explicitly, in the form of monogenic generators.  Is it possible that these special values provide the best power basis for general partial torsion fields?

Our main method involves two ingredients:  the algorithm of Gu\`ardia, Montes and Nart \cite{gmn12}, which computes $[\mathcal{O}_{\QQ(\theta)}:\ZZ[\theta]]$; and the $p$-adic valuations of division polynomials (in particular, $T^4 - 6T^2 - \alpha T - 3$, the $3$-division polynomial in Fueter form), which are computed in detail in work of the third author \cite{s16}.   A basic description of the Montes algorithm is to be found in Section \ref{sec:montes}.  Briefly, the algorithm uses the Newton polygon to compute $v_p([\mathcal O_{\QQ(\theta)}:\ZZ[\theta]])$ in terms of the number of lattice points on and under the polygon. The simplest case is a polygon which bounds no points, and this case corresponds to the $p$-adic valuation being 0.  Thus, by picking $\alpha$ so that all the polygons are simple, we ensure that the corresponding field is monogenic.

It is possible to apply the Montes algorithm to the polynomial $T^4 - 6T^2 - \alpha T - 3$ directly, but the computations are rather involved.  This would provide a proof of Theorem \ref{thm:field-family}, but it would not demonstrate the new methods dependent upon interpreting the polynomial as a division polynomial of an elliptic curve.  In particular, the efficient choice of lift $\phi_i$ (see Section \ref{sec:montes}) is guided by the elliptic curve.

One can view this project as part of the study the discriminants of number fields associated with Latt\`es maps.  Briefly, if $\psi\colon E \to E$ is an elliptic curve endomorphism and $\pi\colon E \to \mathbb P^1$ a finite covering, then a rational map $\phi\colon \mathbb P^1 \to \mathbb P^1$ is a Latt\`es map if $\pi \circ \psi = \phi \circ \pi$.  For example, one may take $\psi(P) = [n]P$ and $\pi(x,y) = x$.  The corresponding Latt\`es map has degree $n^2$, and it is from these maps that the division polynomials are derived (see Section \ref{sec:div-pol}). 

The idea to compute the discriminants of number fields associated to Latt\`es maps is motivated by similar computations done for the power maps and Chebyshev polynomials.  These three families of maps---Latt\`es, Chebyshev, and power---are postcritically finite.  Consequently, if $f$ is a member of any one of these families, then the tower of number fields generated by $f^n(x) - c$ is unramified outside a finite set of primes \cite{ch12}.  In some sense this simplifies the computation of the index as only finitely many primes need be analysed.  In the case that $f$ is a Chebyshev or power map, the first author has used the Montes algorithm to compute the field discriminant precisely, and produced infinite towers of monogenic fields \cite{g14, g17}.  In the case of the $n$-division polynomial, we need only consider the primes dividing $n$ and the discriminant of the curve.  The shape of the Newton polygons tend to evolve predictably from one iterate to the next.

\subsection*{Acknowledgements}  The authors are indebted to David Grant, \'{A}lvaro Lozano-Robledo and Joseph H. Silverman for helpful conversations.

\section{The Montes Algorithm} \label{sec:montes}

In this section we give a basic description of the Montes algorithm so that Theorem \ref{th:gmn} is understood.  We refer more interested readers to \cite{gmn12} for the full details.

Let $\Phi \in \ZZ[x]$ be a monic irreducible polynomial whose root $\theta$ generates a number field $K$, and denote by $\mathcal{O}_K$ the ring of integers of $K$. Define $\ind \Phi = [ \mathcal{O}_K : \ZZ[\theta] ]$.  Let $\ind_p \Phi = v_p(\ind \Phi)$ denote the $p$-adic valuation of $\ind \Phi$. The value $\ind_p\Phi$ may be computed as follows.

First, factor $\Phi$ modulo $p$ and write
\begin{align*}
\Phi(x) \equiv \phi_1(x)^{e_1} \cdots \phi_r(x)^{e_r} \pmod p,
\end{align*}
where the $\phi_i \in \ZZ[x]$ are monic lifts of the irreducible factors of $\Phi$ modulo $p$. The algorithm will terminate regardless of the choice of lifts, however this choice may simplify the computations significantly.

For each factor $\phi_i$, there is a unique expression
\begin{align*}
\Phi(x) = a_0(x) + a_1(x)\phi_i(x) + a_2(x)\phi_i(x)^2 + \cdots + a_s(x)\phi_i(x)^s,
\end{align*}
where the $a_j$ are integral polynomials satisfying $\deg a_j < \deg \phi_i$. This expression is called the \emph{$\phi_i$-development} of $\Phi$.

From the $\phi_i$-development, construct the \emph{$\phi_i$-Newton polygon} by taking the lower convex hull of the points
\begin{align} \label{eq:Newton poly}
\left\{\big(j,v_p(a_j(x))\big): 0 \le j \le s\right\},
\end{align}
where $v_p(a_j(x))$ is defined to be the minimal $p$-adic valuation of the coefficients of $a_j(x)$. Only the sides of negative slope are of import, and we call the set of sides of negative slope the \emph{$\phi_i$-polygon}. The set of lattice points under the $\phi_i$-polygon in the first quadrant carries important arithmetic data, and to keep track of these points, we define
\begin{align*}
\ind_{\phi_i}(\Phi) = (\deg \phi_i)\cdot \#\{ (x,y) \in \NN^2: \text{ $(x,y)$ is on or under the $\phi_i$-polygon}\}.
\end{align*}

To each lattice point on the $\phi_i$-polygon, we attach a \emph{residual coefficient}
\begin{align*}
        \res(j) =
\begin{cases*}
        \red(a_j(x)/p^{v_p(a_j(x))}) & if $\big(j,v_p(a_j(x))\big)$ is on the $\phi_i$-polygon\\
        0 & otherwise,
\end{cases*}
\end{align*}
where $\red: \ZZ[x] \to \bbf_p[x]/(\phi_i(x))$ denotes the reduction map modulo $p$ and $\phi_i$.
For any side $S$ of the $\phi_i$-polygon, denote the left and right endpoints of $S$ by $(x_0,y_0)$ and $(x_1,y_1)$, respectively. We define the \emph{degree} of $S$ to be $\deg S = \gcd(y_1-y_0,x_1-x_0)$. In other words, $\deg S$ is equal to the number of segments into which the integral lattice divides $S$. We associate to $S$ a \emph{residual polynomial}
\begin{align*}
R_S(y) = \sum_{i=0}^{\deg S} \res\left(x_0+i\frac{(x_1-x_0)}{\deg S}\right)y^i \in \bbf_p[x]/(\phi_i(x))[y].
\end{align*}
We note that $\res(x_0)$ and $\res(x_1)$ are necessarily non-zero, and in particular, it is always the case that $\deg S = \deg R_S$.

Finally, if $R_S$ is separable for each $S$ of the $\phi_i$-polygon, then $\Phi$ is \emph{$\phi_i$-regular}, and if $\Phi$ is $\phi_i$-regular for each factor $\phi_i$, then $\Phi$ is \emph{$p$-regular}.

\begin{theorem}[Theorem of the index] \label{th:gmn}
We have
\begin{align*}
\ind_p\Phi \ge \sum_{i=1}^r \ind_{\phi_i}(\Phi)
\end{align*}
with equality if $\Phi$ is $p$-regular.
\end{theorem}

\begin{proof}
See \cite[\S 4.4]{gmn12}.
\end{proof}

For our purposes, we need only the following simple corollary.

\begin{proposition}\label{prop:montes-for-us}
  If $\Phi$ is monic, and $v_p(a_0) = 1$ for each $\phi_i$-development, then $\ind_{p} \Phi = 0$.
\end{proposition}

\begin{proof}
  The Newton polygon for each $\phi_i$-development has exactly one side of negative slope to consider, running from $(0,1)$ to $(k_0, 0)$ for some $0 < k_0 \le s$.  Therefore there are no points under or on the segment, and $\Phi$ is $p$-regular.  The result follows from Theorem \ref{th:gmn}.
\end{proof}

\section{Fueter form and curves with a point of order $4$}

The goal of this section is to examine a particular one-parameter family of elliptic curves, namely a normal form for a curve with a rational point of order $4$ (although often called Tate's normal forms, such families of curves with rational $n$-torsion were known in the 19th century).  This family was suggested by experimental data.  In the next section we exhaustively analyse the valuations of special values of division polynomials for this family, describing all situations in which the Montes algorithm can be applied.

\subsection{Tate and Fueter forms}
\label{sec:tate-fueter-form}

Tate's normal form for an elliptic curve with a rational point of order $4$ is given by the Weierstrass form
\begin{equation}
  \label{eqn:tate}
  E: y^2 + (\alpha + 8 \beta) xy + \beta (\alpha + 8 \beta)^2 y = x^3 + \beta (\alpha + 8 \beta) x^2,
\end{equation}
where $\alpha, \beta \in \QQ$.  Though, by a change of coordinates, we may assume that $\alpha, \beta \in \ZZ$ and are coprime.  Up to isomorphism, this is a one-parameter family of curves with $(0,0)$ being a point of order $4$.  The invariants are:
\begin{equation}
  \label{eqn:invariants}
  \Delta = \beta^4(\alpha - 8\beta)(\alpha + 8 \beta)^7, \quad j = \frac{ (\alpha^2 - 48 \beta^2)^3 }{\beta^4(\alpha - 8\beta)(\alpha + 8\beta)}.
\end{equation}
Throughout the remainder of the paper, we will often use $a := \alpha + 8 \beta$ for ease of notation.

Appling the change of coordinates
\begin{equation}
  \label{eqn:change-of-coord}
  (x,y) = \left( \frac{a \beta}{T} - a\beta, \frac{1}{2} \left( \frac{ (a\beta)^{\frac32}T_1}{T^2} - \frac{a^2\beta}{T} \right) \right),
\end{equation}
one obtains
\[
  T_1^2 = T\left(4T^2 + \frac{\alpha}{\beta} T + 4\right),
\]
which is known as a Fueter curve \cite{ct87}.  The identity of the group is $(T,T_1) = (0,0)$, and the point $Q_0 := (1,\sqrt{a/\beta})=(1,\sqrt{8+\alpha/\beta})$ is a point of order $4$.  Note that this change of coordinates is defined over a potentially quadratic extension $\QQ(\sqrt{a\beta})$ but that the field of definition of the $x$-coordinate of a point is the same as the field of definition of the corresponding $T$ coordinate.

Suppose $p$ is a prime at which $E$ has bad reduction. If $p \mid a$ or $p \mid \beta$, then the singular point modulo $p$ on the Weierstrass curve, namely $(0,0)$, becomes $Q_0$ modulo $p$ on the Fueter curve. However, if $p \mid (\alpha-8\beta)$, then the singular point modulo $p$ on the Weierstrass curve, namely $(-2^5\beta^2,2^7\beta^3)$, becomes $(-1,0)$ modulo $p$ on the Fueter curve. Generally, when $p$ is an odd prime that divides $\alpha-8\beta$, a rational lift of the singular point will not necessarily exist.

\subsection{Division polynomials, Weierstrass and Fueter}
\label{sec:div-pol}

By definition, the $n$-th division polynomial $\Psi_n(x,y)$ for an elliptic curve $E$ in Weierstrass form
\[
  E: y^2 + a_1xy + a_3 = x^3 + a_2x^2 + a_4x + a_6
\]
has the property that
\[
  [n](x,y) = \left( \frac{\phi_n(x,y)}{\Psi_n(x,y)^2}, \frac{\omega_n(x,y)}{\Psi_n(x,y)^3} \right),
\]
where $\phi_n, \omega_n, \Psi_n$ are coprime polynomials.  It can also be defined by stipulating that $\Psi_1(x,y) = 1, \Psi_2(x,y) = 2y + a_1x + a_3$ and for $n > 2$,
\[
  \Psi_n(x,y) = \begin{dcases*}
    n \sideset{}{'}\prod_{P \in E[n]\setminus\{\mathcal{O}\}} (x - x(P)) & $n$ is odd \\
    \frac{n}{2} \Psi_2(x,y) \sideset{}{'}\prod_{P \in E[n]\setminus E[2]} (x - x(P)) & $n$ is even,
  \end{dcases*}
\]
where the $'$ on the product indicates that we include only one of each pair $P$ and $-P$ in the product.  In particular,
\begin{align*}
&\Psi_1=1,\\
&\Psi_2=2y+a_1x+a_3,\\
&\Psi_3=3x^4+b_2x^3+3b_4x^2+3b_6x+b_8,\\
&\Psi_4=\Psi_2(2x^6+b_2x^5+5b_4x^4+10b_6x^3+10b_8x^2+(b_2b_8-b_4b_6)x+(b_4b_8-b_6^2)).
 \end{align*}
The odd division polynomials have degree $\frac{n^2-1}{2}$ in $x$.  The $n$-th division polynomial has divisor $\sum_{P \in E[n]} (P) - n^2 (\mathcal{O})$.  The group law of the elliptic curve manifests as a recurrence relation among the $\Psi_n$, $\omega_n$ and $\phi_n$; in particular, for $n \ge 3$,
\begin{equation}
  \label{eqn:psi-rec}
  \Psi_{2n-1} = \Psi_{n+1}\Psi_{n-1}^3 - \Psi_{n-2}\Psi_{n}^3, \quad
  \Psi_{2n} \Psi_2 = \Psi_n \left( \Psi_{n+2}\Psi_{n-1}^2 - \Psi_{n-2}\Psi_{n+1}^2 \right).
\end{equation}
Therefore, having computed the first four division polynomials directly, we can obtain all the others recursively.

The discriminants of division polynomials have been computed by Verdure:
\begin{theorem}[{\cite[Theorem 1]{Verdure}}]
  \label{thm:psi-disc}
  \[
    \operatorname{Disc}(\Psi_n) =
    \begin{cases*}
      (-1)^{\frac{n-1}{2}} n^{\frac{n^2-3}{2}} \Delta^{ \frac{n^4 - 4n^2 + 3}{24} }  & $n$ odd \\
      (-1)^{\frac{n-2}{2}} 16 n^{\frac{n^2-6}{2}} \Delta^{ \frac{n^4 - 10n^2 + 24}{24} } & $n$ even.
    \end{cases*}
  \]
\end{theorem}

In \cite{Fueter}, Fueter defined similar polynomials in $T$ and $T_1$ which we will call \emph{Fueter polynomials}.  In particular, for a Fueter curve $T_1^2 = T(4T^2 + \frac{\alpha}{\beta}T + 4)$, one defines $F_1 = 1, F_2 = \frac{T_1}{\sqrt{T}}$, and for $n > 2$, 
\[
  F_n = \begin{dcases*}
    \sideset{}{'}\prod_{P \in E[n]\setminus\{\mathcal{O}\}} (T - T(P)) & $n$ is odd \\
    \frac{n}{2} F_2 \sideset{}{'}\prod_{P \in E[n]\setminus E[2]} (T - T(P)) & $n$ is even.
  \end{dcases*}
\]
Here the above products are taken over the nontrivial $n$-torsion points with distinct $T$-coordinates. We also exclude the 2-torsion from the product when $n$ is even. The first few Fueter polynomials are:
\begin{align*}
&F_1=1,\\
&F_2=\frac{T_1}{\sqrt{T}},\\
&F_3=T^4-6T^2-\frac{\alpha}{\beta}T-3,\\ 
&F_4 = 2\frac{T_1}{\sqrt{T}}\left(T^6+\frac{\alpha}{\beta}T^5+10T^4-10T^2-\frac{\alpha}{\beta}T-2\right).
\end{align*}
Furthermore, they satisfy a recurrence relation:
\begin{align}
  \label{eqn:f-rec}
  F_{2n-1} &= (-1)^{n}(F_{n+1}F_{n-1}^3 - F_{n-2}F_{n}^3), \quad \notag \\
  F_{2n} F_2 &= (-1)^nF_n \left( F_{n+2}F_{n-1}^2 - F_{n-2}F_{n+1}^2 \right).
\end{align}
Our Fueter polynomials for odd $n$ coincide with those defined by Cassou-Nogu\`{e}s and Taylor in \cite[\S IV.3]{ct87}. However, our even Fueter polynomials are distinct. In making our definition, we wished to preserve the recurrence relation. 

One now observes that for odd $n$ (our primary interest), the polynomials $\Psi_n(x)$ and $F_n(T)$ define the same field extension.  We will refer to this field extension as the \emph{$n$-th partial torsion field}.  When $n$ is prime, it is the field of definition of the $x$-coordinate or $T$-coordinate of a single point of order $n$, which is generically of degree $(n^2-1)/2$.

Although we will only require the following proposition for odd $n$, we record the full relationship between the division polynomials of the Weierstrass and Fueter forms.
\begin{proposition}
  \label{prop:psi-to-f}
  Let $n$ be odd.  Then
  \[
    \Psi_n = (-1)^\frac{n-1}{2} \left( \frac{a\beta}{T} \right)^{\frac{n^2-1}{2}} F_n,
  \]
  where $F_n$ is a monic polynomial in $T$ of degree $\frac{n^2-1}{2}$.

  Let $n$ be even.  Then
  \[
    \Psi_n = (-1)^\frac{n+2}{2} \left( \frac{a\beta}{T} \right)^{\frac{n^2-1}{2}} F_n,
  \]
  where $F_n=\frac{n}{2}\frac{T_1}{\sqrt{T}}f_n$ with $f_n$ a monic polynomial in $T$ of degree $\frac{n^2-4}{2}$.
\end{proposition}

\begin{proof}
Using the change of coordinates \eqref{eqn:change-of-coord}, we check the result directly for $n=1,2,3,4$.
Proceeding by induction, suppose we have the result for all $n<N$ and consider $\Psi_N$. 

{\bf Case I: $N$ odd.} In this case, letting $N=2m+1$, we have by \eqref{eqn:psi-rec} that
$$\Psi_N=\Psi_{2m+1}=\Psi_{m+2}\Psi_m^3-\Psi_{m-1}\Psi_{m+1}^3.$$
Suppose $m$ is even. Then, using \eqref{eqn:f-rec} and the inductive hypothesis,
\begin{align*}
  \Psi_N&=-\Big(\frac{a\beta}{T} \Big)^{\frac{(m+2)^2-1+3m^2-3}{2}}F_{m+2}F_m^3-\Big(\frac{a\beta}{T} \Big)^\frac{(m-1)^2-1+3(m+1)^2-3}{2}F_{m-1}F_{m+1}^3. \\
&=-\Big(\frac{a\beta}{T} \Big)^{\frac{(2m+1)^2-1}{2}}\left(F_{m+2}F_m^3 + F_{m-1}F_{m+1}^3\right) \\
&=-\Big(\frac{a\beta}{T} \Big)^{\frac{N^2-1}{2}}F_N. 
\end{align*}
Keeping in mind the relationship $T_1^2=4T^3+\frac{\alpha}{\beta}T^2+4T$, we remark that $F_N$ is a polynomial in $T$. 
Finally, %$F_{m+2}F_m^3$ has degree $\frac{4m^2+4m-4}{2}$, so 
the leading term of $F_N(T)$ is determined by $F_{m-1}F_{m+1}^3$, which has degree $(N^2-1)/2$ and is monic.

An analogous computation yields the result if $m$ is odd.

{\bf Case II: $N$ even.}  Letting $N=2m$, we have from \eqref{eqn:psi-rec} that
$$\Psi_2\Psi_N=\Psi_2\Psi_{2m}=\Psi_{m-1}^2\Psi_m\Psi_{m+2}-\Psi_{m-2}\Psi_m\Psi_{m+1}^2.$$
Again suppose $m$ is even. We have from \eqref{eqn:f-rec} and the inductive hypothesis that
\begin{align*}
\Psi_2\Psi_N&=-\Big(\frac{a\beta}{T} \Big)^{\frac{2(m-1)^2-2+m^2-1+(m+2)^2-1}{2}}F_{m-1}^2F_m F_{m+2}\\
&\quad\quad\quad+\Big(\frac{a\beta}{T} \Big)^{\frac{(m-2)^2-1+m^2-1+2(m+1)^2-2}{2}}F_{m-2} F_m F_{m+1}^2 \\
&=-\Big(\frac{a\beta}{T} \Big)^{\frac{(2m)^2+2}{2}}(F_{m-1}^2F_m F_{m+2} - F_{m-2} F_m F_{m+1}^2) \\
&=-\Big(\frac{a\beta}{T} \Big)^{\frac{N^2+2}{2}}F_2F_N.
\end{align*}
Dividing by $\Psi_2=\frac{(a\beta)^\frac{3}{2}T_1}{T^2}$ we obtain our desired expression. Note that 
\begin{align*}
&F_{m-1}^2F_m F_{m+2} - F_{m-2} F_m F_{m+1}^2\\
&=\frac{T_1^2}{T}
\left(
\frac{m^2+2m}{4}F_{m-1}^2f_m f_{m+2} - \frac{m^2-2m}{4}f_{m-2} f_m F_{m+1}^2 
\right).
\end{align*}
The quantity in the large parentheses is a polynomial in $T$, which, by induction, has leading term of degree $(N^2-4)/2$ with coefficient $m$. Finally, as before, if $m$ is odd, an analogous computation finishes the proof.
\end{proof}

We also record the discriminant of the odd Fueter polynomials.

\begin{proposition}\label{prop:f-disc}
  For $n$ odd, we have
  \[
    \operatorname{Disc}(F_n) =         (-1)^{\frac{n-1}{2}} n^{\frac{n^2-3}{2}} \left( \beta^{-2}(\alpha -8\beta)(\alpha + 8\beta) \right)^{ \frac{ n^4 - 4n^2 +3}{24} } .
  \]
\end{proposition}

\begin{proof}
  To compute the discriminant, we use Proposition \ref{prop:psi-to-f}.  Let $d = (n^2-1)/2$, the degree of $\Psi_n$.  Let $n$ be odd.  Then,
\begin{align*}
  \Disc F_n(T) &= (a\beta)^{-2d(d-1)} \Disc (a\beta)^d F_n(T) \\
  &= (a\beta)^{-2d(d-1)} \Disc \left( \Psi_n\left( \frac{a\beta}{T} - a\beta \right) T^d \right) \\
  &= (a\beta)^{-2d(d-1)} \Disc ( \Psi_n( a\beta T - a\beta ) ) \\
  &= (a\beta)^{-d(d-1)} \Disc ( \Psi_n( T - a\beta ) ) \\
  &= (a\beta)^{-d(d-1)} \Disc ( \Psi_n( T ) ).
\end{align*}
Next, we use the discriminant of $E$ \eqref{eqn:invariants} and Theorem \ref{thm:psi-disc}.
\end{proof}

\subsection{Tate's algorithm}

The purpose of this subsection is to give a full analysis of the reduction of the curve $E$ in Tate's Weierstrass form, via Tate's algorithm.

\begin{proposition}
  \label{prop:tate-alg}
  Let $p$ be an odd prime, $p \mid \Delta$.  Let $\widetilde{E}$ denote the reduction of $E$ modulo $p$.  Let $f$ denote the exponent of $p$ in the conductor of $E$.  Let $c$ be the number of components in the special fiber over the minimal proper regular model of the curve over $\ZZ_p$.  Then: 
  \begin{enumerate}
    \item If $p \mid \beta$, then $f=1$, $c = 4 v_p(\beta)$, and $E$ has Kodaira type $I_{4v_p(\beta)}$.  In this case, $E$ is in minimal Weierstrass form with respect to $p$, and the point $(0,0)$ has singular reduction.
    \item If $p \mid (\alpha - 8 \beta)$, then $f = 1$ and $E$ has Kodaira type $I_{v_p(\alpha - 8\beta)}$.  Furthermore,
      \begin{enumerate}
	\item If $p \equiv 1 \pmod 4$, then $c = v_p(\alpha - 8\beta)$.
	\item If $p \equiv 3 \pmod 4$, then
	  \[
	    c = \begin{cases*}
	      1 & if $v_p(\alpha - 8\beta)$ is odd \\
	      2 & if $v_p(\alpha - 8 \beta)$ is even.
	    \end{cases*}
	  \]
      \end{enumerate}
      In these cases, $E$ is in minimal Weierstrass form with respect to $p$, and the point $(-2^5 \beta^2, 2^7 \beta^3)$ on $\widetilde{E}$ is singular.
    \item If $p \mid (\alpha + 8 \beta)$, we let $w=\lfloor\frac{v_p(\alpha+8\beta)}{2}\rfloor$. Then
      \begin{enumerate}
	\item If $v_p(\alpha + 8 \beta)$ is odd, then $f=2$, $c=4$, and $E$ has Kodaira type $I^*_{v_p(\alpha + 8 \beta)}$.
	\item If $v_p(\alpha + 8 \beta)$ is even, then $f=1$, $E$ has Kodaira type $I_{v_p(\alpha + 8 \beta)}$, and
	  \[
	    c = \begin{cases*}
	      v_p(\alpha + 8\beta) & if $\left( \frac{ \beta(\alpha + 8 \beta)p^{-2w} }{p} \right) = 1$ \\
	      2 &  if $\left( \frac{ \beta(\alpha + 8 \beta)p^{-2w} }{p} \right) = -1$.
	    \end{cases*}
	  \]
      \end{enumerate}
      When $p \mid (\alpha+8\beta)$, $E$ is in minimal Weierstrass form with respect to $p$ after the change of coordinates $(x,y)=(p^{2w}x',p^{3w}y')$ and the point $(0,0)$ has singular reduction.
  \end{enumerate}
\end{proposition}

\begin{proof}
  We follow Tate's algorithm as described in \cite[IV \S 9]{ATAEC}.

  {\bf Case I:}  Suppose $p \mid \beta$. We apply Tate's algorithm and note that $p\dnd b_2=(\alpha+8\beta)^2+4\beta(\alpha+8\beta)$. Hence we have Kodaira type $I_{4v_p(\beta)}$ and $f=1$. Since $T^2-\alpha T$ splits completely over $\ZZ/p\ZZ$, $c=4v_p(\beta)$.

  {\bf Case II:}  Suppose $p \mid (\alpha -8\beta)$. In this case the singular point on the reduced curve is $(-2^5\beta^2,2^7\beta^3)$. Following Tate's algorithm, we make a change of coordinates $(x',y')=(x-2^5\beta^2,y+2^7\beta^3)$. Recall the notation $a=(\alpha+8\beta).$ For ease of notation we will write $x'$ as $x$ and $y'$ as $y$. We now have
\begin{align*}
E' &: y^2+axy+(2^8\beta^3+2^5\beta^2a+\beta a^2)y \\
&= x^3+(-3\cdot 2^5\beta^2+\beta a)x^2+(-2^6\beta^3a-2^7\beta^3a+3\cdot2^{10}\beta^4)x \\
& \quad +(-2^7\beta^4a^2+5\cdot2^{10}\beta^5a-3\cdot 2^{14}\beta^6).
\end{align*}
Continuing, we compute $b_2=a_1^2+4a_2$. Note $a\equiv 2^4\beta$ mod $p$. We have
\[
b_2=a^2+2^2(3x_1+\beta a)\equiv 2^8\beta^2-3\cdot2^7\beta^2+2^6\beta^2=-2^6\beta^2.
\]
This shows that $p\dnd b_2$ so that we have Kodaira type $I_{v_p({\alpha-8\beta})}$ and $f=1$. Continuing, we consider $T^2+aT+(3\cdot 2^5\beta^2 -\beta a)$ over $\ZZ/p\ZZ$. Reducing we have $T^2+2^4\beta T+5\cdot 2^4\beta^2$. Applying the quadratic formula, the roots are $-8u\beta \pm 4\beta \sqrt{-1}$. Thus the splitting field is $\ZZ/p\ZZ$ if and only if $p\equiv 1 \mod 4$. Hence $c=v_p(\alpha-8\beta)$ if $p\equiv 1 \mod 4$. Further, if $p\equiv 3 \mod 4$, then $c=1$ if $v_p(\alpha-8\beta)$ is odd and $c=2$ if $v_p(\alpha-8\beta)$ is even.

{\bf Case III:} Now assume $p \mid (\alpha+8\beta)$. Recall $w=\lfloor\frac{v_p(\alpha+8\beta)}{2}\rfloor$. We make the change of coordinates $(x,y)=(p^{2w}x',p^{3w}y')$. We have $a_1\mapsto a_1p^{-w}$, $a_2\mapsto a_2p^{-2w}$, and $a_3\mapsto a_3p^{-3w}$. Note $\Delta'=\Delta p^{-12w}$ so that $v_p(\Delta')=7v_p(\alpha-8\beta)-12w=v_p(\alpha-8\beta)$.

{\bf Part a:}
Suppose $v_p(\alpha+8\beta)$ is odd. Applying Tate's algorithm, we see $p \mid b_2'=(a_1p^{-w})^2+4a_2p^{-2w}$, $p^3 \mid b_8'=a_2a_3^2p^{-8w}$, and $p^3 \mid b_6'=a_3^2p^{-6w}$. Hence we consider $T^3-a_2p^{-v}T^2$ over $\ZZ/p\ZZ$. This polynomial has a double root at $T=0$ and a simple root at $T=a_2p^{-2w}$. Thus we have Kodaira type $I^*_{v_p(\alpha+8\beta)}$ and $f=2$. Following the subprocedure to step 7, we find $c=4$. 

{\bf Part b:} Suppose $v_p(\alpha+8\beta)$ is even. Applying Tate's algorithm, we see that $p\dnd b_2'=(a_1p^{-w})^2-4a_2p^{-2w}$. Hence we have Kodaira type $I_{v_p(\alpha+8\beta)}$ and $f=1$. Considering $T^2-\beta(\alpha+8\beta)p^{-2w}$ over $\ZZ/p\ZZ$, we see that if $\left(\frac{\beta(\alpha+8\beta)p^{-2w}}{p}\right)=1$, then $c=v_p(\alpha+8\beta)$. Conversely, if $\left(\frac{\beta(\alpha+8\beta)p^{-2w}}{p}\right)=-1$ then $c=2$.
\end{proof}

Care must be taken when $E$ has bad reduction at 2. When $2 \mid \beta$, the results and proof used above can be applied by replacing $p$ with 2. When $2 \mid (\alpha+8\beta)$ we see $2 \mid \alpha$ and hence $2 \mid \alpha-8\beta$. 

\begin{proposition}
  \label{prop:tate-alg2}
Let the notation be as before and recall, $a=\alpha+8\beta$.
\begin{enumerate}
\item If $v_2(a)=1$, then $E$ has Kodaira type $I_1^*$, $f=3$, and $c=4$. In this case, $E$ is in minimal Weierstrass form with respect to 2 and the point $(0,0)$ has singular reduction. 
\item If $v_2(a)=2$, then $E$ has Kodaira type $III$. 
\item If $v_2(a)$ is odd and greater than 1, the $E$ has Kodaira type $I^*_{v_2(a)}$. 
\item If $v_2(a)=4$ and $\frac{\beta a+4a-16}{32}$ is odd, then $E$ has Kodaira type $I_0^*$. 
\item If $v_2(a)=4$ and $\frac{\beta a+4a-16}{32}$ is even, then we have two subcases. 
  \begin{enumerate}
    \item If $\frac{\beta a^2}{2^8}\equiv 1$ mod 4, then $E$ has Kodaira type $I_2^*$. 
    \item If $\frac{\beta a^2}{2^8}\equiv 3$ mod 4, then $E$ has Kodaira type $I_3^*$. 
  \end{enumerate}
\item If $v_2(a)>4$ is even, we have several subcases:
  \begin{enumerate}
    \item If $\frac{\beta a+4a-16}{32}$ is odd, then we have Kodaira type $I_{v_2(a)-4}^*$. 
    \item If $\frac{\beta a+4a-16}{32}$ is even, we have further subcases:
      \begin{enumerate}
	\item If $v_2(a)=6$, we have Kodaira type $III^*$. 
	\item If $v_2(a)=8$, then $E$ is nonsingular at 2. 
	\item If $v_2(a)\geq 10$, we have Kodaira type $I_{v_2(a)-8}$.
      \end{enumerate}
  \end{enumerate}
\end{enumerate}
\end{proposition}

\begin{proof}
  We follow Tate's algorithm as described in \cite[IV \S 9]{ATAEC}.

  {\bf Case I: $v_2(a)=1$.} Applying Tate's algorithm, we see $2 \mid b_2$, $4 \mid a_6$, $8 \mid b_8$, and $8 \mid b_6$. Thus we consider $$P(T)=T^3+\frac{\beta a}{2}T^2=T^2\left(T+\frac{\beta a}{2}\right).$$
We see $P(T)$ has a simple root and a double root modulo 2. Hence we have Kodaira type $I_n^*$ and $f=v_2(\Delta)-4-n$. To determine $n$ and $c$ we consider the polynomial
$$Y^2+\frac{\beta a^2}{4}Y.$$
This polynomial has distinct roots in $\ZZ/2\ZZ$. Hence $n=1$ and $c=4$. Noting $v_2(\Delta)=8$, the result follows.

{\bf Case II: $v_2(a)>1$.}  We define $w=\lfloor \frac{v_2(a)}{2}\rfloor$ and we make the change of coordinates $(x,y)=(2^{2w}x',2^{3w}y')$. For ease of notation we will write $x$ and $y$ for $x'$ and $y'$. 

{\bf Case II-A: $v_2(a)=2$.}  Then $8\dnd b_8$ and we have type $III$.

{\bf Case II-B:  $v_2(a)$ odd.}
If $v_2(a)$ is odd we consider $P(T)\equiv T^2(T+1)$ mod 2. When the subprocedure to step 7 terminates, we are left with type $I^*_{v_2(a)}$.

{\bf Case II-C: $v_2(a)=4$.}  In step 6 we change coordinates to obtain
$$y^2 +\left(\frac{a}{2^w}+2\right)xy+\frac{\beta a^2}{2^{3w}}y=x^3+\left(\frac{\beta a}{2^{2w}}+\frac{a}{2^w}-1\right)x^2+\frac{\beta a^2}{2^{3w}}x.$$
We consider $$P(T)=T^3+\frac{\beta a+2^wa-2^{2w}}{2^{2w+1}}T^2+\frac{\beta a^2}{2^{3w+2}}T.$$
If $\frac{\beta a+2^wa-2^{2w}}{2^{2w+1}}$ is odd, then we have type $I_0^*$. If $\frac{\beta a+2^wa-2^{2w}}{2^{2w+1}}$ is even, we change coordinates setting $x=x'+2$ and again abuse notation by letting $x=x'$. Our curve becomes
\begin{align*}
  y^2 &+\left(\frac{a}{2^3}+2\right)xy+\left(\frac{\beta a^2}{2^{3w}}+\frac{a}{2^{w-1}}+4\right)y \\
  &=x^3+\left(\frac{\beta a+2^wa-2^{2w}+6\cdot 2^{2w}}{2^{2w}}\right)x^2+\left(\frac{\beta a^2}{2^{3w}}+\frac{\beta a +2^wa-2^{2w}}{2^{2w}}+12\right)x.
\end{align*}
Following the subprocedure to step 7, we obtain the desired result.

{\bf Case II-D: $v_2(a)>4$ even and $\frac{\beta a+4a-16}{32}$ odd.}  Then $P(T)\equiv T^2(T+1)$ mod 2. Following the subprocedure to step 7, we find we have type $I_{v_2(a)-4}$. 

{\bf Case II-E: $v_2(a)>4$ even and $\frac{\beta a+4a-16}{32}$ even.}  Then $P(T)$ has a triple root.

{\bf Case II-E-i:  $v_2(a)=6$ and $\frac{\beta a+4a-16}{32}$ even.}  Then $16\dnd a_4=\frac{\beta a^2}{2^{3w}}$ so we have type $III^*$.

{\bf Case II-E-ii: $v_2(a) > 6$ even and $\frac{\beta a+4a-16}{32}$ even.}  Then our Weierstrass equation was not minimal. We make the change of coordinates $(x,y)=(4x',8y')$ to obtain
$$y^2 +\left(\frac{a}{2^{w+1}}+1\right)xy+\frac{\beta a^2}{2^{3w+3}}y=x^3+\frac{\beta a+2^wa-2^{2w}}{2^{2w+2}}x^2+\frac{\beta a^2}{2^{3w+4}}x.$$

{\bf Case II-E-ii-a: $v_2(a)=8$ and $\frac{\beta a+4a-16}{32}$ even.} One checks that if $v_2(a)=8$, our curve is nonsingular at 2. 

{\bf Case II-E-ii-b: $v_2(a)>8$ even and $\frac{\beta a+4a-16}{32}$ even.} We have type $I_{v_2(a)-8}$.
\end{proof}

\section{Valuation of Division Polynomials}

The purpose of this section is to determine the valuation of $F_n$ evaluated at the singular point.  This is done by reference to the valuations of $\Psi_n$ at the singular point, and the change of variables of Proposition \ref{prop:psi-to-f}.  To obtain the valuations of $\Psi_n$, we demonstrate two methods.  The first is to apply the results of \cite{s16}, which give explicit valuations based on the reduction data of Proposition \ref{prop:tate-alg}.  The second is a hands-on approach using the recurrence relations for division polynomials, which is possible in simpler cases.  We consider only odd primes.

\subsection{Odd primes dividing $\alpha - 8 \beta$}

Recall that, when $p \mid (\alpha - 8\beta)$, the singular point modulo $p$ is $(-2^5\beta^2, 2^7\beta^3)$.

\begin{proposition}
  \label{prop:val-minus}
  Suppose $p \mid (\alpha - 8\beta)$.  Let $Q$ be a point of $E(\overline{\mathbb{Q}})$ which is singular modulo $p$, and satisfies $x(Q) = -2^5\beta^2$. Let $Q'$ be the image of $Q$ under the change of coordinates to Fueter form.
  Suppose that $n$ is odd.  Then,
  \[
    v_p( F_n(Q')) = v_p(\Psi_n(Q)) = 
      v_p(\alpha - 8 \beta) \frac{n^2-1}{8}.  
  \]
\end{proposition}

To prove Proposition \ref{prop:val-minus}, we begin with a lemma.

\begin{lemma}  \label{lem:2Q}
  Suppose $p \mid (\alpha - 8\beta)$ and let $Q$ be as above. Then, $[2]Q$ does not reduce to the singular point mod $p$.
\end{lemma}

\begin{proof} Recall $a=\alpha+8\beta$. We compute 
  \begin{align*}
    x([2]Q)&=\frac{2^{20}\beta^8-b_42^{10}\beta^4+b_62^6\beta^2-b_8}{-2^{17}\beta^6+b_22^{10}\beta^4-b_42^6\beta^2+b_6}\\
    &=\frac{2^{20}\beta^6-2^{10}a^3\beta^3+2^6a^4\beta^2-a^5\beta}{-2^{17}\beta^4+2^{10}a^2\beta^2+2^{12}a\beta^3-2^6a^3\beta+a^4}.
  \end{align*}
We divide the numerator and denominator by $a-16\beta=\alpha-8\beta$ to obtain
\[
\frac{-a^4\beta+3\cdot2^4a^3\beta^2-2^8a^2\beta^3-2^{12}a\beta^4-2^{16}\beta^5}{a^3-3\cdot2^4a^2\beta+2^8a\beta^2+2^{13}\beta^3}.
\]
Reducing mod $p$ we obtain
\[
x([2]Q)\equiv -2^4\beta^2.
\]
Thus $[2]Q$ does not reduce to the singular point.
\end{proof}

Following \cite{s16}, we define, for any integers $a, \ell$ such that $\ell \neq 0$, the sequence
\begin{equation}
  \label{eqn:rn}
  R_n(a,\ell) = \left\lfloor \frac{ n^2\widehat{a}(\ell-\widehat{a}) }{2\ell} \right\rfloor - \left\lfloor \frac{ \widehat{na}(\ell - \widehat{na}) }{2\ell} \right\rfloor ,
\end{equation}
where $\widehat{x}$ denotes the least non-negative residue of $x$ modulo $\ell$.  Theorem 9.3 of \cite{s16} gives the valuations of the sequence of division polynomials, evaluated at a point of multiplicative reduction, in terms of such sequences.  We apply this to our specific situation here.

In particular, we will encounter the sequence $R_n(1,2)$, which begins from $n=1$ as follows:
\[
0, 1, 2, 4, 6, 9, 12, 16, 20, 25, 30, 36, 42, \ldots
\]
The odd terms of the sequence have a simple closed form.
\begin{lemma}
  \label{lem:triang}
  For $n$ odd, $R_n(1,2) = \frac{n^2-1}{4}$.
\end{lemma}

\begin{proof}
  For $n$ odd, we have $\widehat{a}=\widehat{na}=1$ in \eqref{eqn:rn}.  Therefore,
  \[
  R_n(1,2) = \left\lfloor \frac{ n^2 }{4} \right\rfloor - \left\lfloor \frac{ 1 }{4} \right\rfloor
  = \left\lfloor \frac{ n^2 }{4} \right\rfloor = \frac{n^2-1}{4}.
\]
\end{proof}

\begin{proposition}
  \label{prop:unified}
  Suppose $p \mid (\alpha - 8\beta)$ and let $Q$ be as above. Let $K$ be the potentially quadratic extension of $\QQ$ so that $Q\in E(K)$ and let $L$ be an unramified, potentially quadratic extension of $K$ such that $E$ has split multiplicative reduction over $L$ (which exists by Proposition \ref{prop:tate-alg}).  Let $v_p'$ be a lift of $v_p$ to $L$.  
Let $n>0$ and suppose $4 \dnd n$.  Then $v_p' = 2v_p$ if and only if $v_p(\alpha - 8\beta)$ is odd; otherwise $v_p' = v_p$.
We have
$$v_p'(\Psi_n(Q))=\frac{v_p'(\alpha-8\beta)}{2}R_n(1,2).$$
  If furthermore $n$ is odd, then
  $$v_p(\Psi_n(Q)) = v_p(\alpha - 8\beta)\frac{n^2-1}{8}.$$
\end{proposition}

\begin{proof}
  One can compute that $K$ is the quadratic extension obtained by adjoining 
  \begin{align*}
    &\sqrt{\alpha^4-2^5\alpha^3\beta-2^7\alpha^2\beta^2+5\cdot 2^{11}\alpha\beta^3-15\cdot 2^{12}\beta^4} \\
    &= \sqrt{\alpha -8 \beta}\sqrt{ \alpha^3 - 24\alpha^2\beta - 320 \alpha \beta^2 +7680 \beta}.
  \end{align*}
  We also have
  \[
    \alpha^3 - 24\alpha^2\beta - 320 \alpha \beta^2 + 7680 \beta \equiv 2^{12}\beta^3 \pmod{\alpha - 8\beta}.
  \]
  Therefore, since $p$ is odd, divides $(\alpha - 8\beta)$, and is coprime to $\beta$, we have that the extension $K$ is ramified at $p$ if and only if $v_p(\alpha - 8\beta)$ is odd.  Hence, $v_p' = 2 v_p$ if and only if $v_p(\alpha - 8\beta)$ is odd; otherwise $v_p' = v_p$.  

  The group of components over $L$ is isomorphic to $\ZZ/v_p'(\alpha-8\beta)\ZZ$ since we have split multiplicative reduction. The component containing $Q$ has additive order exactly 2 by Lemma \ref{lem:2Q}. Thus it may be identified with $v_p'(\alpha-8\beta)/2$.  
  Hence, in the language of \cite{s16}, $\ell_Q=v_p'(\alpha-8\beta)$ and $a_Q=v_p'(\alpha-8\beta)/2$. Applying \cite[Theorem 9.3]{s16}, we find that
$$v_p'(\Psi_n(Q))=R_n(v_p'(\alpha-8\beta)/2,v_p'(\alpha-8\beta)).$$
  By \cite[Proposition 8.2(iv)]{s16},
  $$v_p'(\Psi_n(Q))=\frac{v_p'(\alpha-8\beta)}{2}R_n(1,2). $$

For odd $n$, $\Psi_n(x)$ is a polynomial in $x$ alone and therefore $\Psi_n(Q) \in \QQ$.  Accordingly, by Lemma \ref{lem:triang}, we obtain the given statement.
\end{proof}

Proposition \ref{prop:val-minus} follows from Propositions \ref{prop:unified} and \ref{prop:psi-to-f} (recall that $\alpha, \beta$ are coprime integers).

\subsection{Odd primes dividing $\alpha + 8 \beta$ or $\beta$}

In this case, we apply the recurrence relation for the division polynomial to obtain valuations.

\begin{proposition}
  \label{prop:val-others}
  Suppose $p \mid \beta$ or $p \mid (\alpha + 8 \beta)$ (these cases are mutually exclusive).  Then $(0,0)$ is a point of order $4$ and has singular reduction on $\widetilde{E}$; the corresponding point in Fueter form has $T=1$.  Suppose that $n$ is odd.

  If $p \mid \beta$, then
   \[
    v_p(\Psi_n(0)) = \frac{3n^2 -3}{8} v_p(\beta), \quad
    v_p(F_n(1)) = -\frac{n^2 -1}{8} v_p(\beta).
  \]

  If $p \mid (\alpha + 8\beta)$, then
 \[
    v_p(\Psi_n(0)) = \frac{5n^2 - 5}{8} v_p(\alpha + 8 \beta), \quad
    v_p(F_n(1)) = \frac{n^2 -1}{8} v_p(\alpha + 8\beta).
  \]
\end{proposition}

\begin{proof}
We will proceed by induction. Recall $a=\alpha+8\beta$. For the base cases we have $\Psi_1(0)=1$, $\Psi_2(x,y)=2y+ax+a^2$ so $\Psi_2(0)=a^2$. Further, $\Psi_3=3x^4+b_2x^3+3b_4x^2+3b_6x+b_8=3x^4+(a^2+4\beta a)x^3+3\beta a^3x^2+3\beta^2a^4x+\beta^3a^5$. Hence $\Psi_3(0)=\beta^3a^5$. We have $\Psi_4=\Psi_2(2x^6+b_2x^5+5b_4x^4+10b_6x^3+10b_8x^2+(b_2b_8-b_4b_6)x+(b_4b_8-b_6^2))$. Evaluating at 0 we obtain $\Psi_4(0)=\Psi_2(0)(b_4b_8-b_6^2)=\Psi_2(0)(\beta^4a^8-\beta^4a^8)=0$.

First we prove if $4 \mid n$, $\Psi_n(0)=0$. Suppose we have the result for all $n<N$ and suppose $4 \mid N$.  Let $N=2m$, so that $m$ is even.  Then
$$\Psi_2\Psi_N=\Psi_2\Psi_{2m}=\Psi_{m-1}^2\Psi_m\Psi_{m+2}-\Psi_{m-2}\Psi_m\Psi_{m+1}^2.$$
Now either $4 \mid m$ or $4 \mid m-2$ and $4 \mid m+2$. Hence the result follows by induction.

Now suppose that $v_p(\Psi_n(0))=v_p(a)\frac{5n^2-5}{8}+v_p(\beta)\frac{3n^2-3}{8}$ for all $n<N$. Suppose $N$ is odd, and write $N=2m+1$. We have
$$\Psi_N=\Psi_{2m+1}=\Psi_{m+2}\Psi_m^3-\Psi_{m-1}\Psi_{m+1}^3.$$
Suppose first that $m$ is even.  Then either $m$ or $m+2$ is divisible by 4. Hence
\begin{align*}
v_p(\Psi_N(0))&=v_p(\Psi_{m-1}(0))+3v_p(\Psi_{m+1}(0))\\
&=v_p(a)\frac{5(m-1)^2-5}{8}+v_p(\beta)\frac{3(m-1)^2-3}{8}\\
&\quad\quad\quad\quad+v_p(a)3\frac{5(m+1)^2-5}{8}+v_p(\beta)3\frac{3(m+1)^2-3}{8}\\
&=v_p(a)\frac{5(2m+1)^2-5}{8}+v_p(\beta)\frac{3(2m+1)^2-3}{8}.
\end{align*}

Likewise, if $m$ is odd, either $m-1$ or $m+1$ is divisible by 4. Hence
\begin{align*}
v_p(\Psi_N(0))&=v_p(\Psi_{m+2}(0))+3v_p(\Psi_{m}(0))\\
&=v_p(a)\frac{5(m+2)^2-5}{8}+v_p(\beta)\frac{3(m+2)^2-3}{8}\\
&\quad\quad\quad\quad+v_p(a)3\frac{5m^2-5}{8}+v_p(\beta)3\frac{3m^2-3}{8}\\
&=v_p(a)\frac{5(2m+1)^2-5}{8}+v_p(\beta)\frac{3(2m+1)^2-3}{8}.
\end{align*}

This gives the stated results for $\Psi_n$.  For $F_n$, we use the change of coordinates between Weierstrass and Fueter form and Proposition \ref{prop:psi-to-f}.
\end{proof}

\section{Proof of the Main Theorem}

\begin{proof}[Proof of Theorem \ref{thm:ec-family}]
  Suppose $E$ is an elliptic curve defined over $\QQ$, and suppose a twist $E'$ has a rational $4$-torsion point, hence can be put into Tate normal form as in \eqref{eqn:tate} with $\alpha, \beta \in \ZZ$ coprime.  The $j$-invariant of the elliptic curve is invariant under twisting.  In Tate normal form, the discriminant and $j$-invariant are of the form
  \[
    \Delta = \beta^4(\alpha - 8\beta)(\alpha + 8\beta)^7, \quad j = \frac{ (\alpha^2 - 48 \beta^2)^3 }{ \beta^4 (\alpha - 8\beta)(\alpha + 8\beta) }, \quad \alpha,\beta \in \ZZ.
  \]
  Therefore $E'$ has good reduction modulo $p$ unless $p \mid \beta(\alpha - 8\beta)(\alpha + 8\beta)$.
 
  We now show that conditions \eqref{item:red}, \eqref{item:sqf} and \eqref{item:j} of the statement are equivalent.  Under condition \eqref{item:red}, we have $\beta=1$ by Proposition \ref{prop:tate-alg}.  In this case, requirements \eqref{item:sqf} and \eqref{item:j} are evidently equivalent.  For odd primes, Proposition \ref{prop:tate-alg} implies that $p^2$ may not divide $\alpha \pm 8$.  For $p=2$, Proposition \ref{prop:tate-alg2} implies that $v_2(\alpha + 8) = 0$ or $1$.  This implies $v_2(\alpha - 8) = 0$ or $1$ also, and we have demonstrated condition \eqref{item:j}.  Hence \eqref{item:red} implies \eqref{item:sqf} and \eqref{item:j}.
  Conversely, if condition \eqref{item:sqf} holds, we apply Propositions \ref{prop:tate-alg} and \ref{prop:tate-alg2} to conclude that \eqref{item:red} holds.  Thus we have demonstrated all the conditions are equivalent.

   %Therefore reduction on $E'$ modulo $p$ is good unless $p \mid \beta(\alpha - 8\beta)(\alpha + 8\beta)$.  The requirement that $\beta = 1$ and $\alpha \pm 8$ be squarefree is equivalent to requiring that $v_p(j) \ge -1$.  In this situation, via Tate's algorithm (Proposition \ref{prop:tate-alg}), all reduction on $E'$ at odd primes is of type $I_1$ or $I_1^*$, and furthermore there is no bad reduction at $2$ (since $2 \nmid \Delta$ under squarefreeness).  Conversely, if all reduction is of type $I_1$ or $I_1^*$, and good at $2$, then again via Proposition \ref{prop:tate-alg}, $\beta^4(\alpha-8\beta)(\alpha + 8\beta)$ is not divisible by any square.  We have therefore shown that the three conditions of the theorem are equivalent.

  The field $K_\alpha$ generated by the $x$-coordinate of a single point of order $3$ is invariant under the twist.  Therefore we now assume $E$ itself has a rational $4$-torsion point.
  Change coordinates so that $E$ is in Tate normal form and Fueter form as in Section \ref{sec:tate-fueter-form} with $\alpha \in \ZZ$ and $\beta=1$.  We then find that the partial $3$-torsion field is generated by the $3$-division Fueter polynomial, $F_3(T) = T^4 - 6T^2 - \alpha T - 3$.  Let $\theta$ be a root of this polynomial, and let $K = \QQ(\theta)$.
  Under the equivalent conditions of the theorem, the polynomial $F_3(T)$ is irreducible, as observed in \cite[Proposition 2.10]{FV}, so $K$ is a quartic field.

  We apply the Montes algorithm.  It calls for examining the polynomial $F_3$ developed around any lift of a repeated irreducible factor modulo $p$; each such situation may contribute a factor to the index $[\mathcal{O}_K: \ZZ[\theta]]$.  If no such non-trivial factors appear, we can conclude $\mathcal{O}_K = \ZZ[\theta]$.

  We will show prime-by-prime that the only repeated factors are linear of the form $T-T_0$ and that $v_p(F_3(T_0)) = 1$.

  {\bf Case I: $p=2$.}  Modulo $2$, the polynomial $F_3$ becomes $T^4 - \alpha T - 1$.  If $\alpha$ is odd, this is irreducible with no repeated roots.  If $\alpha$ is even, then the repeated root is $1$, so we develop $F_3$ around $T-1$, obtaining a constant term of $-\alpha - 8$, which we have assumed to be squarefree.  Therefore in this case $v_2(F_3(1)) = 1$.

  {\bf Case II: $p=3$.}  
  Modulo $3$, the polynomial $F_3$ becomes $T^4 - \alpha T$, and $\alpha$ is a repeated root. If $3$ divides $\alpha$, then a lift of this root is $0$, and $v_3(F_3(0)) = 1$.  If $\alpha \equiv 1 \pmod 3$, then $4$ is a lift, and $v_3(F_3(4)) = 1$. Else $-4$ is a lift of $\alpha$, and $v_3(F_3(-4)) = 1$.

  {\bf Case III: $p \ge 5$.}
  Now, suppose $F_3$ has a repeated irreducible factor modulo an odd prime $p$.  The roots of $F_3$ are the four $x$-coordinates of non-trivial $3$-torsion; this means that reduction modulo $p$ fails to be injective on $E[3]$.  This occurs if and only if $E$ has bad reduction at $p$, or $p=3$.

  Suppose $p \ge 5$ is a prime of bad reduction, and suppose $Q$ is a point on $E$ having singular reduction modulo $p$.  Specifically, if $p \mid \alpha + 8$, take $Q = (0,0)$.  If $p \mid \alpha - 8$, take $x(Q) = -2^5$.  Then, the only repeated root of $F_3$ modulo $p$ is $T(Q)$ (since the failure of injectivity under reduction must take the form of $3$-torsion points mapping to the singular point, as the map to the non-singular part has torsion-free kernel).  Then, using the fact that $\alpha \pm 8$ are not divisible by $p^2$, we learn from Propositions \ref{prop:val-minus} and \ref{prop:val-others} that $v_p(F_3(T(Q))) = 1$.

  In each case, we find that $v_p(F_3(T_0))=1$ where $T_0$ is the repeated root.  Therefore the associated Newton polygon starts at height $1$ on the $y$-axis.  Hence, the polygon cannot pass through any lattice points and cannot contain any lattice points, and the polygon has only one segment, as in Proposition \ref{prop:montes-for-us}.  Therefore it is $p$-regular.  By the Montes algorithm, this implies that the index $[\mathcal{O}_K: \ZZ[\theta]]$ is not divisible by $p$.

  As we have verified that the index $[\mathcal{O}_K : \ZZ[\theta]]$ is not divisible by any prime, we conclude that $\mathcal{O}_K = \ZZ[\theta]$.
\end{proof}

Theorem \ref{thm:field-family} follows immediately.

\section{Algebraic number theory of the family $T^4 - 6T^2 - \alpha T - 3$}\label{sec:ant}

Let $\theta$ be a root of $T^4 - 6T^2 - \alpha T - 3$.  Consider the field $K_\alpha = \QQ(\theta)$.
This family of number fields was studied by Fleckinger and V\'erant \cite{FV}.  Let $\alpha \ge 9$, $\alpha \in \ZZ$, and $\alpha \neq 24$.  Then Fleckinger and V\'erant showed that $K_\alpha$ is an $S_4$ quartic field with two real embeddings \cite[Proposition 2.10]{FV}.  They give an explicit basis for the ring of integers in general \cite[Proposition 2.11]{FV}, but it is not a power basis and they do not mention monogenicity.  Finally, they remark that when $3 \mid \alpha$, then $1 + \frac{\alpha}{3} \theta + 2\theta^2$ is a unit.  In fact, they point out that there are no other parametrized units in this field.
Experimentally, we observed surprisingly small regulators and surprisingly large class groups for these fields; the existence of a simple parametrized unit is a possible explanation.

\section{A related family}

Fleckinger and V\'erant also study the family of quartic fields given by $T^4 + \frac{\alpha}{2} T^3 + 6 T^2 + \frac{\alpha}{2} T + 1$ of discriminant $-4\left( \left(\frac{\alpha}{2}\right)^2 - 16 \right)^3$, which they observe arise from a point of order four on a Fueter model \cite{FV}.  The authors prove that this family is monogenic whenever $(\alpha/2)^2 - 16$ is odd and squarefree, and $\alpha \ge 12$ \cite[Corollary 1.4]{FV}.  This appears to be a $D_8$ family.  We leave it as an open question whether the methods of this paper may apply to this family.

\section{Experimental Data}
\label{sec:exp}

As part of our exploration, we took a survey of elliptic curves to determine the prevalence of monogenic fields, using Sage Mathematics Software \cite{Sage} and pari/GP \cite{PARI}.  Up to isogeny, there are 11575 curves of conductor less than 10000 whose 3-division field is monogenic.  The torsion points of many curves share the same field of definition, and in all, these 11575 curves yield 1026 unique fields.  In particular, the following families of fields are prevalent.
\begin{center}
\renewcommand{\arraystretch}{1.2}
\begin{tabular}{l|l} 
\text{Polynomial} & \text{Discriminant}\\
\hline 
$T^4 - 6s T^2 - t T - 3s^2$ & $-3^3(t^2 - 64s^3)^2$ \\
$T^4 - T^3 - 3s T^2 - (4t + 3s^2)T + t$ & $-3^3(16t^2 + (24s^2 + 12s + 1)t + (9s^4 + s^3))^2$\\
$T^4 - 2T^3 - 6s T^2 - (2t+6s^2)T + t$ & $-2^4 3^3 (t^2 + (6s^2 + 6s + 1)t + (9s^4 + 2s^3))^2$
\end{tabular}
\end{center}
In the table above $T$ is the indeterminate, while $s,t \in \ZZ$ parametrize the family.  Each of these quartic field families appears to be $S_4$ monogenic under appropriate conditions on the discriminant and the parameters.

\bibliography{MonoDivBib}
\bibliographystyle{plain}

\end{document}